\documentclass[rmfamily,a4paper,10pt]{article}
\usepackage{theorem,amssymb,amsmath,german,oldgerm}

\usepackage[all]{xy}

\newcommand{\be}{\begin{enumerate}}
\newcommand{\ee}{\end{enumerate}}
\newcommand{\bi}{\begin{itemize}}
\newcommand{\ei}{\end{itemize}}
\newcommand{\beq}{\begin{equation}}
\newcommand{\eeq}{\end{equation}}

\newcommand{\ra}{\mbox{$\rightarrow$}}

\newcommand{\Leqa}{\mbox{$\Longleftrightarrow$}}

\newcommand{\tens}[3]{\mbox{${#1}\otimes_{#2}{#3}$}}

\newcommand{\lattices}[1]{\mathop{\text{lattices}({#1})}}
\newcommand{\rad}[1]{\mathop{\text{rad}({#1})}}
\newcommand{\Her}[1]{\mathop{\text{Her}({#1})}}
\newcommand{\Span}[2]{\mathop{\text{span}_{#1}({#2})}}
\newcommand{\Centr}[2]{\mathop{\text{C}_{#1}({#2})}}
\newcommand{\LC}{\mathop{\text{LC}}}
\newcommand{\End}[2]{\mathop{\text{End}_{#1}(#2)}}
\newcommand{\Aut}[2]{\mathop{\text{Aut}_{#1}(#2)}}

\newcommand{\diag}{\mathop{\text{diag}}}
\newcommand{\EMatr}[1]{\mathop{\text{I}_{#1}}}
\newcommand{\Matr}[2]{\mathop{\text{M}_{#1}({#2})}}
\newcommand{\Mint}[3]{\mathop{\text{M}_{#1,#2}(#3)}}

\newcommand{\pairs}{\mathop{\text{pairs}}}

\newcommand{\pD}{\mbox{$\textfrak{p}_D$}}
\newcommand{\pF}{\mbox{$\textfrak{p}_F$}}

\newcommand{\trans}[1]{{#1}^{\ensuremath{\mathsf{T}}}}

\DeclareMathOperator{\Gal}{Gal}
\DeclareMathOperator{\Latt}{Latt}

\DeclareMathOperator{\row}{row}
\DeclareMathOperator{\MopRow}{Row}

\DeclareMathOperator{\MopEnd}{End}
\DeclareMathOperator{\period}{period}
\DeclareMathOperator{\rank}{rank}
\newenvironment{example}{\text{\bf Example:}}{}
\newenvironment{proof}{\text{\bf Proof:} }{\text{q.e.d.}}

\newtheorem{theorem}{Theorem}
\newtheorem{proposition}{Proposition}
{\theorembodyfont{\rmfamily}
\newtheorem{notation}{Notation}
\newtheorem{definition}{Definition}
\newtheorem{definitionremark}{Definition/Remark}

\newtheorem{lemma}{Lemma}

\newtheorem{remark}{Remark}}

\entrymodifiers={!!<0pt,0.7ex>+}

\title{Embeddings of local fields in simple algebras and simplicial structures on the Bruhat-Tits building}
\date{29.08.2008}
\author{Daniel Skodlerack}

\begin{document}
\maketitle
\section{Introduction and notation}
\subsection{First remark}
This article answers a question that naturally arises from the articles by M.Grabitz and P. Broussous (see \cite{broussousGrabitz:00}) and P. Broussous and B. Lemaire (see \cite{broussousLemaire:02}). 
For an Azumaya-Algebra $A$ over a non-archimedean local field $F,$ M. Grabitz and P. Broussous have introduced embedding invariants for field embeddings, that is for pairs $(E,\textfrak{a})$, where $E$ is a field extension of $F$ in $A$, and $\textfrak{a}$ is a hereditary order which is normalised by $E^{\times}.$ On the other hand if we take such a field extension $E$ and define $B$ to be the centralizer of $E$ in $A,$ then $G:=A^{\times}$ are $G_{E}:=B^{\times}$ are sets of rational points of reductive groups defined over $F$ and $E$ respectively. P. Broussous and B. Lemaire have defined a map $j_E:\ {\cal I}^{E^\times}\rightarrow {\cal I}_E$, where ${\cal I}$ is the g.r. (geometric realization) of the euclidean building of $G$, and ${\cal I}_E$ is the g.r. of the euclidean building of $G_{E}.$
The question which we address is to relate the embedding invariants to the behavior of the map $j_E$ with respect to the simplicial structures of ${\cal I}$ and ${\cal I}_E.$ I have to thank very much Prof. Zink from Homboldt University Berlin for his helpful remarks, the revision of the work and for giving my the interesting task.

\subsection{Notation}
\be
\item The set of natural numbers starts with 1 and the set of the first $r$ natural numbers is denoted by ${\mathbb{N}}_{r}.$ For the set of non-negative integers we use the symbol $\mathbb{N}_0.$
\item The letter $F$ denotes a non-archimedean local field.  For the valuation ring, the valuation ideal, the prime element and the residue field of $F$ we use the notation $o_F, \pF, \pi_F$ and $\kappa_F$ respectively. We use similar notation for other division algebras with non-archimedean valuation. 
\item The letter $\nu$ denotes the valuation on $F$ with $\nu(\pi_F)=\frac{1}{q},$ where $q$ is the cardinality of $\kappa_F.$
\item We assume $D$ to be a finite dimensional central division algebra over $F$ of index $d.$ 
\item We fix an $m$ dimensional right $D$ vector space $V$, $m\in\mathbb{N},$ and put $A:=\End{D}{V}.$ In particular $V$ is a left $\tens{A}{F}{D^{op}}$-module.
\item The letter $L$ denotes a maximal unramified field extension of $F$ in $D$ and we assume that $\pi_D$ normalizes $L,$ i.e. the map $\sigma(x):=\pi_Dx\pi_D^{-1},$ $x\in D,$ generates $\Gal(L|F)$. 
\item For a positive integer $f|d$ we denote by $L_f$ the subfield of degree $f$ over $F$ in $L.$
\ee

\section{Preliminaries}

\subsection{Vectors and Matrices up to cyclic permutation}\label{secVectorsMatrices}
\begin{remark}
All invariants which are considered in this aritcle are vectors or matrices modulo cyclic permutation. 
\end{remark}
{\bf Vectors:} We denote by $\MopRow(s,t)$ the set of all vectors $w\in\mathbb{N}_0^s$ whose sum of entries is $t,$ where $s$ and $t$ are natural numbers, i.e.
$$\sum_{i=1}^sw_i=t.$$
Two vectors $w,w'\in\MopRow(s,t)$ are called {\it equivalent} if $w$ can be obtained from $w'$ by cyclic permutation of the entries of $w,$ i.e.
$$w'=(w_k,\ldots,w_s,w_1,\ldots,w_{k-1})\text{ for a }k\in\mathbb{N}_s.$$
The equivalence class is denoted by $\langle w\rangle.$ Analogous we define the class for every vector, e.g. for a vector of pairs.
One can represent the class $\langle w\rangle$ of a vector $w\in\MopRow(s,t)$ by pairs 
$$\pairs(\langle w\rangle):=\langle (w_{i_0},i_1-i_0),(w_{i_1},i_2-i_1),\ldots,(w_{i_{k}},i_0+m'-1-i_k)\rangle ,$$
where $(w_{i_j})_{0\leq j\leq k}$ is the subsequence of the non-zero coordinates.   
Given the same $w$ with $\pairs(\langle w\rangle)=\langle (a_0,b_0),\ldots,(a_k,b_k)\rangle $
we define the {\it complement of} $\langle w\rangle$, denoted by $\langle w\rangle^c$ to be the class $\langle w\rangle$,
such that $$\pairs(\langle w\rangle)=\langle (b_0,a_1),(b_1,a_2),(b_2,a_3),\ldots,(b_k,a_0)\rangle .$$ 
This is a bijection $$(\ )^c:\ \MopRow(s,t)\ra \MopRow(t,s).$$ \vspace{3mm} \\
{\bf Matrices:} Given three natural numbers $s,r,t$ the symbol $\Mint{r}{s}{t}$ denotes the set of matrices with $r$ rows and $s$ columns such that all entries are non-negative integers and the sum of them is $t.$ 
For a matrix $M=(m_{i,j})\in\Mint{r}{s}{t},$ we define the vector $\row(M)\in\MopRow(r+s,t)$ to be 
$$(m_{1,1},m_{1,2},\ldots,m_{1,s},m_{2,1},\ldots,m_{2,s},\ldots,m_{r,s}).$$
Two matrices $M,N\in\Mint{r}{s}{t}$ are said to be {\it equivalent} if $\row(M)$ and 
$\row(N)$ are. The equivalence class is denoted by $\langle M\rangle .$

\begin{example}
\[\left(\begin{array}{cc}2&0\\1&3\\0&1\end{array}\right)\backsim\left(\begin{array}{cc}1&2\\0&1\\3&0\end{array}\right) 
\]
\end{example}

\subsection{Hereditary orders, lattice chains and lattice functions}
In this section we give a description of the euclidean building of $A$ in terms of lattice functions and its simplicial 
structure in terms of hereditary orders. As references we recommand \cite{reiner:03} for hereditary orders , \cite{brown:89} for the definition of an euclidean building   and \cite{broussousLemaire:02} for the discription of the euclidean building of $GL_m(D)$ in terms of lattice functions and norms. One can find the description with norms in \cite{bruhatTitsIII:84} as well. For more details see also \cite{skodlerack:05}.  

\begin{definition}
Let $D'$ be a central division algebra of finite index $d'$ over a non-archimedian local field $F',$ and let $W$ be a $D'$-vector space of finite dimension. A finitely generated $o_{D'}$-submodule $\Gamma$ of $W$ is called a {\it (full) $o_{D'}$ lattice of} $W$ if $\Span{D'}{\Gamma}=W.$ We omit the word full.
\end{definition}

\begin{definition}
A subring $\textfrak{a}$ of $A$, is called an $o_F$-{\it order of} $A$ if  $\textfrak{a}$ is an $o_F$-lattice of $A.$
We call an $o_F$-order $\textfrak{a}$ {\it hereditary} if  the Jacobson radical $\rad{\textfrak{a}}$ is a projective right-$\textfrak{a}$ module. The set of all hereditary orders is denoted by $\Her{A}.$
For $\textfrak{a}\in\Her{A}$ we denote by $\lattices{\textfrak{a}}$ the set of all $o_D$-lattices $\Gamma$ of $V$ such that $a\Gamma\subseteq\Gamma$ for all $a\in\textfrak{a}.$
\end{definition}

\begin{definitionremark}
\be
\item Let $R$ be a non-empty set, and take $r\in\mathbb{N}$. Given non-empty subsets $R_{i,j}$ of $R,$ $(i,j)\in{\mathbb{N}}_{r}^2,$ and natural numbers $n_1,\ldots,n_r$, we denote by $(R_{i,j})^{n_1,\ldots,n_r}$ the set of all block matrices in $M_{\sum_{i=1}^rn_i}(R)$, such that for all $(i,j)$ the \mbox{$(i,j)$}-block lies in $M_{n_i,n_j}(R_{i.j}).$ 
\item Given $r\in\mathbb{N},$ $\bar{n}=(n_1,\ldots,n_r)\in\mathbb{N}^r,$ we get a hereditary order $$\textfrak{a}^{\bar{n}}:=(R_{i,j})^{n_1,\ldots,n_r},\text{ where }$$  $$R_{i,j}:=\left\{\begin{array}{ll}o_D,&\text{if}\ j\leq i\\\pD,&\text{if}\ i<j\end{array}\right. .$$ 
\item A hereditary order of $M_{m}(D)$ of this form is called {\it in standard form.} The class $\langle \bar{n}\rangle $ is called the {\it invariant} and $r$ the {\it period} of $\textfrak{a}^{\bar{n}}.$
\ee
\end{definitionremark}
If we say that sets are conjugate to each other, we mean conjugate by an element of $A^\times.$ The proof is given in \cite{reiner:03}.

\begin{theorem}
We fix a $D$-basis of $V$ and identify $A$ with $M_m(D).$
\be
\item Two hereditary orders in standart form of $A$ are conjugate to each other if and only if they have the same invariant.
\item Every $\textfrak{a}\in\Her{A}$ is conjugate to a hereditary order in standard form. 
\ee
\end{theorem}

By this theorem the notion of {\it invariant} and {\it period} carries over to every element of $\Her{A}$ and they do not depend on the choice of the basis. 

\begin{definition}
A sequence $(\Gamma_i)_{i\in\mathbb{Z}}$ of lattices of $V$ is called an {\it $o_F$-lattice chain in} $V$ if  
\be
\item for all integers $i,$ we have $\Gamma_{i+1}\subset\Gamma_{i},$ and
\item there exists a natural number $r$ such that for all integers $i$ we have $\Gamma_i\pi_D=\Gamma_{i+r}.$  
\ee
We call $r$ the {\it period of the lattice chain.} 
For a lattice chain $\Gamma$ we put
$$\lattices{\Gamma}:=\{\Gamma_i|\ i\in\mathbb{Z}\}.$$
Two lattice chains $\Gamma,\ \Gamma'$ are called {\it equivalent} if  $\lattices{\Gamma}$ and $\lattices{\Gamma'}$ are equal. We write $[\Gamma]$ for the equivalence class. We define an order by $[\Gamma]\leq [\Gamma']$ if  $\lattices{\Gamma}$ is a subset of $\lattices{\Gamma'}.$ The set of all lattice chains in $V$ is denoted by $\LC_{o_D}(V).$ 
\end{definition}

\begin{remark}
For every lattice chain $\Gamma$ in $V,$ the set 
$$\textfrak{a}_{\Gamma}:=\{a\in A|\ \forall i\in\mathbb{Z}:\ a\Gamma_i\subseteq\Gamma_i\}$$ 
is a hereditary order of $A.$
\end{remark}

\begin{theorem}\cite[(1.2.8)]{bushnellFroehlich:83}
$[\Gamma]\mapsto \textfrak{a}_{\Gamma}$ defines a bijection between the set of equivalence classes of lattice chains in $V$ and the set of hereditary orders of $A.$ We have:
$$[\Gamma]\leq[\Gamma']\ \Leqa\ \textfrak{a}_{\Gamma}\supseteq\textfrak{a}_{\Gamma'}$$
\end{theorem}

\begin{definition}
A family $(\Lambda(t))_{t\in\mathbb{R}}$ of $o_D$-lattices of $V$ is called an {\it ($o_D$)-lattice function of} $V$ if the following are satisfied for all $s\leq t:$
\be
\item $\Lambda(t)\subseteq\Lambda(s),$
\item $\bigcap_{r<t}\Lambda(r)=\Lambda(t),$ and
\item $\Lambda(t)\pi_D=\Lambda(t+\frac{1}{d}).$
\ee  
Two lattice functions $\Lambda$ and $\Lambda'$ are called {\it equivalent} if  there is an $s\in\mathbb{R}$ such that for all $t\in\mathbb{R},$  $\Lambda(t)=\Lambda'(t+s).$ We write $[\Lambda]$ for the equivalence class. 
We denote by $\Latt^1_{o_D}(V)$ the set of all $o_D$-lattice functions in $V,$ by $\Latt_{o_D}(V)$ the set of equivalence classes of elements of $\Latt^1_{o_D}(V)$ and by $\lattices{\Lambda}$ the image of $\Lambda,$ i.e.
$$\lattices{\Lambda}=\{\Lambda(t)|\ t\in\mathbb{R}\}.$$
\end{definition}

\begin{definition}\label{latticesquarelatticefunctions}
\bi
\item For every lattice function $\Lambda$ of $V$ we can define a family $t\mapsto\textfrak{a}_{\Lambda}(t):=\{a\in A|\ \forall s\in\mathbb{R}:\ a\Lambda(s)\subseteq\Lambda(s+t)\}.$ $\textfrak{a}_{\Lambda}$ is an $o_F$-lattice function in the $F$-vector space $A.$
We call the set $$\Latt^2_{o_F}(A):=\{\textfrak{a}_{\Lambda}|\ \Lambda\in\Latt^1_{o_F}(V)\}$$ the set of {\it square lattice functions in} $A.$
\item $[\Lambda]\mapsto\textfrak{a}_{\Lambda}$ defines a bijection between $\Latt_{o_D}(V)$ and $\Latt^2_{o_F}(A).$
\ei
\end{definition}

\section{Embedding types}
For a field extension $E|F$ we denote by $E_D|F$ the maximal field extension in $E|F,$ which is $F$-algebra isomorphic to a subfield of $L.$ Its degree is the greatest common divsor of $d$ and the residue degree of $E|F.$  
\begin{definition}
An {\it embedding} is a pair $(E,\textfrak{a})$ satisfying 
\be
\item $E$ is a field extension of $F$ in $A$,
\item $\textfrak{a}$ is a hereditary order in $A$, normalised by $E^{\times}.$ 
\ee
Two embeddings $(E,\textfrak{a})$ and $(E',\textfrak{a}')$ are said to be {\it equivalent} if 
there is an element $g\in A^{\times}$, such that $gE_Dg^{-1}=E'_D$ and $g\textfrak{a}g^{-1}=\textfrak{a}'$.
\end{definition}

\begin{remark}
In each equivalence class of embeddings there is a pair such that the field can
be embedded in $L$. 
\end{remark}

\begin{notation}
For $r,s,t\in\mathbb{N},$ $\Mint{r}{s}{t}$ denotes the set of $r\times s$-matrices
  with non-negative integer entries, such that
\bi
\item in every column there is an entry greater than zero, and
\item the sum of all entries is $t$.
\ei
\end{notation}

Until the end of this section we fix a $D$-basis of $V$ and identify $A$ with $M_m(D).$
\begin{definition}
Let $f|d$ and $r\leq m$.
A matrix with $f$ rows and $r$ columns is called an embedding datum if it belongs to $\Mint{f}{r}{m}$. Given an embedding datum $\lambda$, we define the pearl embedding as follows. 
The {\it pearl embedding} of $\lambda$ (with respect to the fixed basis) is the embedding $(E,\textfrak{a})$, with the
following conditions:
\be
\item $[E:F]=f$,
\item $E$ is the image of the field extension $L_f$ of degree $f$ of $F$ in $L$ via the monomorphism $x\in L_f\mapsto \diag(M_1(x),M_2(x),\ldots,M_r(x))$
where
$$M_j(x)=\diag(\sigma^0(x)\EMatr{\lambda_{1,j}},\sigma^1(x)\EMatr{\lambda_{2,j}},\ldots,\sigma^{f-1}(x)\EMatr{\lambda_{f,j}})$$
\item $\textfrak{a}$ is a hereditary order in standard form according to the partition $m=n_1+\ldots +n_r$ where $n_j:=\sum_{i=1}^f\lambda_{i,j}.$
\ee 
\end{definition}

\begin{theorem}\cite[2.3.3 and 2.3.10]{broussousGrabitz:00}\label{thmPearlEmb}
\be
\item Two pearl embeddings are equivalent if and only if the embedding
  data are equivalent.
\item In any class of embeddings lies a pearl embedding.
\ee 
\end{theorem}

\begin{definition}
Let $(E,\textfrak{a})$ be an embedding. By the theorem it is equivalent to a
pearl-embedding. The class of the corresponding matrix $(\lambda_{i,j})_{i,j}$
is called the {\it embedding type} of $(E,\textfrak{a}).$ This definition does not depend on the choice of the basis by the proposition of Skolem-Noether.  
\end{definition}

\section{The euclidean building of $GL_m(D)$}
\subsection{Definitions}
Here we give the basic definitions, to state precisely the description of the euclidean building of $A^\times$ with lattice chains.
Basic definitions of the notions of simplicial complex and chamber coplex are given in \cite[Ch. I App.]{brown:89}.
For the definition of a Coxeter complex see \cite[Thm. III 4.b]{brown:89}.

\begin{definition}
A {\it building} is a triple $(\Omega,{\cal A},\leq)$, such that
$(\Omega,\leq)$ is a simplicial complex and ${\cal A}$ is a set of subcomplexes of $(\Omega,\leq)$ which cover $\Omega,$ i.e.
$$\bigcup {\cal A} =\Omega,$$
(The elements of ${\cal A}$ are called {\it apartments.}) statisfying the following \grqq Building Axioms\grqq:
\bi
\item {\bf B0} Every element of ${\cal A}$ is a Coxeter complex.
\item {\bf B1} For {\it faces} (also called simplicies), i.e. elements, $S_1$ and $S_2$ of $\Omega$ there is an apartment $\Sigma$ containing them. 
\item {\bf B2} If $\Sigma$ and $\Sigma'$ are two arpartments containing $S_1$ and $S_2$ then there is a poset isomorphism from $\Sigma$ to $\Sigma'$ which fixes $\bar{S}_1$ and $\bar{S}_2$ where $\bar{S}$ for a face $S$ is defined to be the set of all faces $T\leq S.$ 
\ei
A building is said to be {\it thick} if every codimension 1 face is attached to at least three chambers. 
\end{definition}

\begin{remark}
The buildings considered in this article are thick. 
\end{remark}

A euclidean Coxeter comlex is a Coxeter complex $(\Sigma,\leq)$ 
which is poset-isomorphic to simplicial complex $\Sigma(W,V)$ defined by 
an essential irreducible infinite affine reflection group $(W,V).$
 
For a face $S$ of a simplicial complex $(\Omega,\leq)$ the set of all formal sums
$\Sigma_{v\leq S,rk(v)=1}\lambda_vv$ with positive real coefficients such that $\Sigma_{v\leq S,rk(v)=1}\lambda_v=1$ is denoted by $|S|.$ The set 
$$|\Omega|:=\bigcup_{S\in\Omega}|S|$$ 
is called the geometric realisation (g.r.) of $\Omega.$
A {\it morphism} of simplicial complexes from $(\Omega,\leq)$ to $(\Omega',\leq')$ is a
map $f:(\Omega,\leq)\ra (\Omega',\leq'),$ such that for every face $S\in\Omega$ the restriction 
$f:\ \bar{S}\ra \bar{f(S)}$ is a poset isomorphism. In \cite{brown:89} the notion of non-degenerate simplicial map is used instead of morphism. A morphism $f$ induces a map $|f|$ between the g.r., by 
$$|f|(\sum_v\lambda_vv):=\sum_v\lambda_vf(v).$$
\begin{definition}
Given two buildings $(\Omega,{\cal A},\leq)$ and $(\Omega',{\cal A}',\leq')$ a {\it morphism}
from the first to the latter is a morphism of simplicial complexes such that the image of an apartment of ${\cal A}$ is contained in an apartment of ${\cal A}'.$
\end{definition}

As descriped in \cite{brown:89} VI.3 there is a canonical way to define a metric, up to a scalar, on the g.r. of an euclidean building by pulling back the metric from an affine reflection group to the apartment and this defines a canonical affine structure on the g.r. of the building.
The map $|\phi|$ between the g.r. of two euclidean buildings induced by an isomorphism $\phi$ is affine. 

\subsection{The description with lattice functions}

We describe the euclidean Bruhat-Tits-building $\Omega$ of $\Aut{D}{V}$ in terms of lattice chains, and the g.r. ${\cal I}$ in terms of lattice functions as it is done in $\cite{broussousLemaire:02},$ section I.3.

\begin{proposition}
\be
\item The posets $(\LC_{o_D}(V),\leq)$ and $(\Her{A},\supseteq)$ are simplicial complexes  of rank $m$ They are isomorphic via $\Psi([\Gamma]):=\textfrak{a}_{\Gamma}$ as simplicial complexes. 
\item A hereditary order is a vertex (resp. a chamber) if and only if its period is 1 (resp. m).
\ee
\end{proposition}

\begin{definition}
A {\it frame of} $V$ is a set of lines $v_1D,\ldots,v_mD$, where $v_i, i\in\mathbb{N}_m,$ is a $D$-basis of $V.$ If $\textfrak{R}$ is a frame we say that a lattice $\Gamma$ {\it is split by}\/ $\textfrak{R}$ if  $$\Gamma=\bigoplus_{W\in\textfrak{R}}(\Gamma\cap W).$$ 
A lattice chain $\Gamma$, lattice function $\Lambda,$ hereditary order $\textfrak{a}$ is {\it split by}\/ $\textfrak{R}$ if  every element of $\lattices{\Gamma}$, $\lattices{\Lambda},$ $\lattices{\textfrak{a}}$ resp. is split by $\textfrak{R}$. An equivalence class is {\it split by}\/ $\textfrak{R}$ if  every element of the equivalence class is split by $\textfrak{R}.$ The set of these classes split by $\textfrak{R}$ is called the {\it apartment corresponding to \textfrak{R}} and is denoted by ${\LC}_{o_D}(V)_{\textfrak{R}},$  $\Her{A}_{\textfrak{R}},$ $\Latt_{o_D}(V)_{\textfrak{R}}$ resp.. For the set of these apartments we write
$$\textfrak{A}({\LC}_{o_D}(V)),\ \textfrak{A}(\Her{A})\ \&\  \textfrak{A}(\Latt_{o_D}(V)).$$
\end{definition}

\begin{definition}
The left action of $A^{\times}$ on the set of $o_D$-lattices of $V,$ i.e.
$$g.\Gamma:=\{g\gamma|\ \gamma\in\Gamma\},$$
defines an $A^{\times}$-action on ${\LC}_{o_D}(V),\ \Latt_{o_D}(V)$ and $\Her{A}.$ 
\end{definition}

\begin{proposition}
\be
\item The two triple $$({\LC}_{o_D}(V),\textfrak{A}({\LC}_{o_D}(V)),\leq)\ \&\  (\Her{A},\textfrak{A}(\Her{A}),\supseteq)$$ are isomorphic euclidean buildings via $\Psi.$
\item $\Psi$ is $A^{\times}$-equivariant.
\item For every frame $\textfrak{R}$ the image of  ${\LC}_{o_D}(V)_{\textfrak{R}}$ under $\Psi$ is $\Her{A}_{\textfrak{R}}.$
\ee
\end{proposition}
For the proof see for example $\cite{reiner:03}.$

\begin{remark}
Every $\textfrak{a}\in\Her{A}$ has a rank as a face in the chamber complex $(\Her{A},\supseteq),$ and we have $\period(\textfrak{a})=\rank(\textfrak{a}).$ This rank is not the $o_F$-rank of $\textfrak{a}$.
\end{remark}

The barycenter of two points in $\Latt_{o_D}(V)$ is defined as follows. 
Take $[\Lambda],[\Lambda']\in\Latt{o_D}{V}.$ Then there exists a 
frame $\textfrak{R}=\{v_1D,\ldots,v_mD\}$ which splits both. We can write
$$\Lambda(s)=\bigoplus_kv_k\textfrak{p}_D^{[(s+\alpha_k)d]+}$$ and
$$\Lambda'(s)=\bigoplus_kv_k\textfrak{p}_D^{[(s+\alpha'_k)d]+}$$ for some real 
vectors $\alpha,\alpha'$. Take $\beta\in\mathbb{R}$, and put 
$$\beta[\Lambda]+(1-\beta)[\Lambda']:=[\Lambda''],$$ where 
$$\Lambda''(s):=\bigoplus_kv_k\textfrak{p}_D^{[(s+\alpha''_k)d]+}$$ with 
$$\alpha'':=\beta\alpha+(1-\beta)\alpha'.$$ 

Now the next propositions explain why one can replace $\Omega$ by the building of classes of lattice functions and ${\cal I}$ by $\Latt_{o_D}V.$

\begin{proposition}[\cite{broussousLemaire:02} I. 2.4 and II. 1.1. for $F=E$]\label{propaffinebij}
There is a unique $A^\times$ affine bijection from ${\cal I}$ to $\Latt_{o_D}V.$ 
\end{proposition}

The g.r. of ${\LC}_{o_D}(V)$ can be identified with $\Latt_{o_D}V$ in the following way. 
We put $$ [x]+:=\inf\{z\in\mathbb{Z}|\ x\leq z\},\ x\in\mathbb{R},$$
and we define a bijective map $$\tau:\ |{\LC}_{o_D}(V)|\ra \Latt_{o_D}(V)$$ as follows.
A convex barycenter $\sum \beta_i[\Gamma^i]$ with vertices $[\Gamma^i]$ of a chamber of ${\LC}_{o_D}(V)$ is mapped to $\sum_i\beta_i[\Lambda^i]$, where $\Lambda^i(t):=\Gamma^i_0\textfrak{p}_D^{[td]+}.$

\begin{proposition}[\cite{broussousLemaire:02} sec. I.3]
The bijection of proposition \ref{propaffinebij} induces an $A^\times$-equivariant isomorphism from $\Omega$ to the building $({\LC}_{o_D}(V),\textfrak{A}({\LC}_{o_D}(V)),\leq).$
\end{proposition}

\begin{notation}
By the two propositions above we can identify $\Omega$ with \mbox{$(\Her{A},\textfrak{A}(\Her{A}),\supseteq)$} and ${\cal I}$ with $\Latt_{o_D}V$ and 
$\Latt^2_{o_F}(A).$
\end{notation}


\section{The map $j_E$}\label{secThemapjE}

\begin{notation}
For this section let $E|F$ be a field extension in $A$ and we set $B$ to be the centraliser of $E$ in $A,$ i.e.
$$B:=C_A(E):=\{a\in A|\ ab=ba\ \forall b\in B\}.$$
The building of $B$ we denote by $\Omega_E$ and its g.r. by ${\cal I}_E.$
\end{notation}
The next results are taken from \cite{broussousLemaire:02}.

\begin{theorem}\cite[Thm 1.1.]{broussousLemaire:02}
There exists a unique application $j_E:{\cal I}^{E^{\times}}\ra {\cal I}_E$ such that for any $x\in {\cal I}$ we have $j_E(\textfrak{a}(x))=B\cap\textfrak{a}(e(E|F)x).$
The map $j_E$ satisfies the following properties:
\be
\item it is a $B^\times$-equivariant bijection.
\item it is affine.
\ee
Moreover its inverse $j_E^{-1}$ is the only map ${\cal I}_E\ra {\cal I}$ such that 2. and 3. hold.
\end{theorem}

We briefly give Broussous and Lemaire's description of $j_E$ in terms of lattice functions but only in the case, where $E|F$ is isomorphic to a subextension $L_f|F$ of $L|F.$ Then $\tens{E}{F}{L}\cong\bigoplus_{k=0}^{f-1}L$ comming from the decomposition $1=\sum_{k=0}^{f-1}1^k$ labeled such that the
$\Gal(L|F)$-action on the second factor gives $\sigma(1^k)=1^{k-1}$ for $k\geq
1$ and $\sigma(1^0)=1^{f-1}$.
Applying it on the $\tens{E}{F}{L}$-module $V$, we get $V=\bigoplus_kV^k$, $V^k:=1^kV.$

\begin{remark}
\be
\item $B\cong\End{\Delta_E}{V^1}$ and
\item $B\cong M_{m}(\Delta_E)$ 
\ee
where $\Delta_E:=\Centr{D}{L_f}.$
\end{remark}

\begin{theorem}\cite[II 3.1.]{broussousLemaire:02}
In terms of lattice functions $j_E$ has the form
$j_E^{-1}([\Theta])=[\Lambda]$, with 
$$\Lambda(s):=\bigoplus_{k=0}^{f-1}\Theta(s-\frac{k}{d})\pi_D^k,\ s\in \mathbb{R}$$ 
where $\Theta$ is an $o_{\Delta}$-lattice function on $V^1.$
\end{theorem}

\section{The connection between embedding types and barycentric coordinates}\label{secidea}
In this sction we keep the notation from section \ref{secThemapjE}.
We repeat that $E_D$ denotes the biggest extension field of $E|F$ which can be embedded in $L|F.$ The centralizer of $E_D$ in $A$ is denoted by $B_D.$ We need a notion of  orientation on $\Omega_{E_D}$ to order the barycentric coordimnates of a point in ${\cal I}_{E_D}.$ 


\begin{definition}
An edge of $\Omega$ with vertices $e$ and $e'$ is {\it oriented towards} $e',$
if there are lattices $\Gamma\in\lattices{e}$ and $\Gamma'\in\lattices{e'},$ such that 
$\Gamma\supseteq\Gamma'$ with the quotient having $\kappa_D$-dimension 1, i.e. 
$\kappa_F$-dimension $d.$ We write $e\ra e'$ If $x$ is a point in ${\cal I}$ then there is a chamber $C\in\Omega$ such that $x$ lies in the closure of $|C|,$ i.e. in
$$\bigcup_{S\leq C}|S|.$$ The vertices of $C$ can be given in the way
$$e_1\ra e_2\ra\ldots\ra e_{m}\ra e_1.$$ 
If $(\mu_i)$ are the barycentric coordinates of $x$ with respect to $(e_i),$ i.e.  
$$x=\sum_i\mu_ie_i,$$
then the class $\langle \mu\rangle $ is called the {\it local type of} $x.$
\end{definition}
This definition applies for ${\cal I}_{E_D}$ as well. The skewfield is then $C_D(E)$ instead of $D$ and one has to substitute $d$ by $\frac{d}{[E_D:F]}.$ 

\begin{proposition}
The notion of local type does not depend on the choice of the chamber $C$ and the starting vertex $e_1.$
\end{proposition}

For the definition of $\langle  \rangle ^c$ see section \ref{secVectorsMatrices}.

\begin{theorem}\label{thmconnection}
Let $(E,\textfrak{a})$ be an embedding of $A$ with embedding type $\langle \lambda\rangle $ and suppose $\textfrak{a}$ to have period $r.$ If $M_{\textfrak{a}}$ denotes the barycenter of $\textfrak{a}$ in ${\cal I}^{E}$ and $\langle \mu\rangle $ the local type of $j_{E_D}(M_{\textfrak{a}}),$ then the following holds.
\be
\item $r[E_D:F]\mu\in \mathbb{N}_0^m,$ and
\item $\langle v(\lambda)\rangle =\langle [E_D:F]r\mu\rangle ^c.$ 
\ee
\end{theorem}

\begin{remark}
With theorem \ref{thmconnection} we can calculate the embedding type from the local type. For example take $r=2,\ f=6,\ m=7$ and  assume that $j_{E_D}(M_{\textfrak{a}})$ is
$$\frac{3}{12}b_0+\frac{2}{12}b_1+\frac{1}{12}b_2+\frac{0}{12}b_3+\frac{0}{12}b_4+\frac{4}{12}b_5+\frac{2}{12}b_6.$$
and thus $$\langle 12\mu\rangle =\langle 3,2,1,0,0,4,2\rangle \equiv\langle (3,1),(2,1),(1,3),(4,1),(2,1)\rangle .$$
From the complement $$\langle 12\mu\rangle ^c\equiv \langle (1,2),(1,1),(3,4),(1,2),(1,3)\rangle \\ \equiv \langle 1,0,1,3,0,0,0,1,0,1,0,0\rangle $$ applying theorem \ref{thmconnection} we can deduce  the embedding type of $(E,\textfrak{a}):$ 
\begin{displaymath}
\left(\begin{array}{cc}1&0\\1&3\\0&0\\0&1\\0&1\\0&0\end{array}\right).
\end{displaymath}
\end{remark}

For the proof we can restrict to the case where $E=E_D$ and thus $B=B_D.$ We put $f:=[E:F],$ i.e.
$$E\cong L_f\subseteq L$$
and 
$$F\subseteq E\subseteq B\subseteq A.$$
 Firstly we need some lemmas. The actions of $G$ on square lattice functions by conjugation induces maps  
$$m_g:\ \Omega \ra \Omega,\ x\mapsto g.x $$
and 
$$c_g:\ {\cal I}_E\ra {\cal I}_{gEg^{-1}},\ y\in \Latt^2_{o_E}B\mapsto gyg^{-1}\in\Latt^2_{o_{gEg^{-1}}}gBg^{-1} $$
for $g\in G.$ 
\begin{lemma}\label{lemmaCommutativeDiagram}
$|m_g|$ and $c_g$ induce isomorphisms on the simplicial structures of the euclidean buildings, which preserve the orientation, i.e. an oriented  edge is mapped to an oriented edge such that the direction is preserved. In particular $|m_g|$ and $c_g$ are affine bijections, $m_g$ preserves the embedding type, $c_g$ the local type, and the following diagram is commutatve:
\begin{displaymath}
\xymatrix{
{\cal I}^{E^\times} \ar[r]^{ |m_g| } \ar[d]^{j_E} & {\cal I}^{(gEg^{-1})^\times} \ar[d]^{j_{gEg^{-1}}}\\
{\cal I}_{E} \ar[r]^{c_g} & {\cal I}_{gEg^{-1}}
}
\end{displaymath}
\end{lemma}  

The following lemma gives a geometric interpretation of the map
$$\langle \trans{\row()}\rangle :\ \{\text{embedding types}\}\ra\{\text{embedding types of vertices}\} $$ 

\begin{lemma}[rank reduction lemma]\label{LemRankReduction}
Assume there is a field extension $F'|F$ of degree $s$ in $E|F,$ where $2\leq s\leq m.$ Let $\textfrak{a}$ be a vertex in $\Omega^{E^\times}$ such that $\textfrak{a}\cap C_A(F')$ is a face of rank $s$ in $\Omega^{E^\times}_{F'}$ and assume $(E,\textfrak{a})$ has embedding type $\langle \lambda\rangle $ and \mbox{$(E,\textfrak{a}\cap \Centr{F'}{A})$} has embedding type $\langle \lambda'\rangle .$ Then we get 
$$\row(\lambda)\sim\row(\lambda'),\text{ i.e. }\lambda\sim\trans{\row(\lambda')}.$$
\end{lemma}

\begin{proof}
By lemma \ref{lemmaCommutativeDiagram} it is enough to show the result only for one embedding equivalent to $(E,\textfrak{a}).$ For simplicity we can restrict ourself to the case of $s=2.$ The argument for $s>2$ is similar. We fix a $D$-basis of $V.$ It is $(E,\textfrak{a})$ equivalent to the pearl embedding $(\lambda)=:(E_\lambda,\textfrak{a}_\lambda).$ Now we apply a permutation $p$ on $(\lambda)$ such that the odd exponents of $\sigma$ in $pE_\lambda p^{-1}$ are behind all even exponents, i.e. $pE_\lambda p^{-1}$ is the image of 
$$x\in L_{2f}\mapsto \diag(M_{n1}(x),M_{n2}(x)),\ n1:=\sum_{i\text{odd}}\lambda_i,
\ n2:=\sum_{i\text{even}}\lambda_i$$
where
$$M_{n1}(x)=\diag(\sigma^0(x)\EMatr{\lambda_{1}},\sigma^2(x)\EMatr{\lambda_{3}},\ldots,\sigma^{2f-2}(x)\EMatr{\lambda_{2f-1}})$$
and
$$M_{n2}(x)=\diag(\sigma^1(x)\EMatr{\lambda_{2}},\sigma^3(x)\EMatr{\lambda_{4}},\ldots,\sigma^{2f-1}(x)\EMatr{\lambda_{2f}}).$$
For the embedding $(E',\textfrak{a}')$ obtained by conjugating $p(\lambda) p^{-1}$ with the matrix 
$$\diag(\EMatr{n1},\pi_D^{-1}\EMatr{n2})$$ we have the following properties.
\bi
\item $F'$ is the image of the diagonal embedding of $L_2$ in $M_m(D)$ and its centraliser is $M_m(\Delta_{F'}),$ where $\Delta_{F'}:=\Centr{D}{L_2}$
\item The intersection of $\textfrak{a}'$ with $M_m(\Delta_{F'})$ is a herditary order in standard form with invariant $\langle n1,n2\rangle .$ The positivity of the integers follows from the assumption that this intersection is a face of rank $2.$
\ei
Since $\pi_{\Delta_{F'}}:=\pi_D^2$ is a prime element of $\Delta_{F'}$ which normalises $L$ and since the powers of $\sigma$ occuring in the description of $E'$ are even we can read the embedding type of $(E',\textfrak{a}'\cap M_m(\Delta_{F'}))$ directly. It is  
the class of
\begin{displaymath}
\left(\begin{array}{cc}
\lambda_1 & \lambda_2\\
\lambda_3 & \lambda_4\\
\vdots &\vdots \\
\lambda_{2f-1} & \lambda_{2f}\\
\end{array}\right).
\end{displaymath}
Thus the result follows.
\end{proof}

The next lemma shows that changing the skewfield does not change the embedding type.

\begin{lemma}[changing skewfield lemma]\label{lemmaChangingSkewfield}
Let $D'$ be a central skewfield over a local field $F'$ of index $d$ with a maximal unramified extension field $L'$ normalized by a prime element $\pi_{D'}$ and assume that $V'$ is an $m$ dimensional right vector space over $D'.$ Denote the euclidean building of $GL_m(D')$ by ${\cal I}'$ and let $\Sigma,\ \Sigma'$ be an apartment of ${\cal I},\ {\cal I}'$ corresponding to a basis $(v_i),\ (v'_i)$ respectively. Then $\Sigma'$ is fixed by the image $E'$ of the diagonal embedding of $L'_f$ in $\Matr{m}{D'}.$ Assume further that $E$ is the image of diagonal embedding of $L_f$ in $\Matr{m}{D}.$ Under these assumptions the map $\equiv$ from $|\Sigma|$ to $|\Sigma'|$ defined by
$$ [x\mapsto\bigoplus_iv_i\textfrak{p}_D^{[d(x+\alpha_i)]+}]\mapsto [x\mapsto\bigoplus_iv'_i\textfrak{p}_{D'}^{[d(x+\alpha_i)]+}] $$
is the g.r. from an isomorphism $\phi$ of simplicial complexes which preserves the orientation and the embedding type. The latter means that if $\textfrak{a}'$ is the image of a hereditary order $\textfrak{a}$ under $\phi$ then the embedding types of 
$(E,\textfrak{a})$ and $(E',\textfrak{a}')$ equal.
\end{lemma}

\begin{proof}
We only show the preserving of the embedding type. The other properties are verified easely. If $\textfrak{L}$ is a lattice chain of 
$\textfrak{a}$ and $\textfrak{L}'$ one of $\textfrak{a}'$ then by $$\phi(\textfrak{a})=\textfrak{a}'$$ we can assume that for all indexes $j$ the lattices $\textfrak{L}_j$ and $\textfrak{L}'_j$ have the same exponent vectors, i.e. 
if 
$$\textfrak{L}_j=\bigoplus_iv_i\textfrak{p}_D^{\nu_{i,j}}$$ then
$$\textfrak{L}'_j=\bigoplus_iv'_i\textfrak{p}_{D'}^{\nu_{i,j}}$$
and thus by applying from the left a permutation matrix $P$ and a diagonal matrix $T,\ T',$ whose entries are powers of the corresponding prime element, we obtain simultanously lattice chains corresponding to hereditary orders $\textfrak{b},\ \textfrak{b}'$ in the same standard form.  More precisely $T'$ is obtained from $T$ if $\pi_D$ is substituted by $\pi_{D'}.$ 
Thus $(TPEP^{-1}T^{-1},\textfrak{b})$ and $(T'PE'P^{-1}T'^{-1},\textfrak{b}')$ have the same embedding type and thus by conjugating back $(E,\textfrak{a})$ and $(E',\textfrak{a}')$ have the same embedding type.
\end{proof}

We now fix a $D$-basis $v_1,\ldots,v_m$ of $V$ and therefore a frame 
$$\textfrak{R}:=\{R_i:=v_iD |\ 1\leq i\leq m\}$$
and an apartment $\Sigma=\Her{A}_{\textfrak{R}}$ of $\Omega.$
The algebra $A$ can be identified with $M_m(D).$
By the affine bijection $|\Sigma| \cong\mathbb{R}^{m-1} $
which maps 
$$ [\Lambda] \text{ with }\Lambda(x)=\bigoplus_i\textfrak{p}_D^{[d(x+\alpha_i)]+}$$
to
$$ d(\alpha_1-\alpha_2,\ldots,\alpha_{m-1}-\alpha_m), $$
we can introduce affine coordinates on $|\Sigma|$ where the points of $|\Sigma|$ corresponding to the vectors $0,\ (f,0,\ldots,0),\ (0,f,0,\ldots,0),\ldots,\ (0,\ldots,0,f)$ are denoted by $Q_1,Q_2,\ldots,Q_m.$ 

\begin{remark}
The vertices of $\Sigma$ are exactly the points of
$$ Q_1+\sum_{i=2}^m\frac{1}{f}\mathbb{Z}(Q_i-Q_1).$$
\end{remark}

\begin{remark}\label{remarkDiagonalEmbedding}
If $E\subseteq\cap_{i=1}^m\MopEnd_DR_i$ then $\Sigma\subseteq\Omega^{E^\times}$ and 
for an element $g\in \cap_{i=1}^m(\MopEnd_DR_i)^\times,$ i.e. a diagonal matrix,  $|m_g|$ induces an affine bijection of $|\Sigma |.$
If $g$ is $\diag(1,\ldots,1,\pi_D^k,1,\ldots,1),$ with $\pi_D^k$ in the i-th row, the map $|m_g|$ is of the form 
$$Q\mapsto Q+\frac{k}{f}(Q_{i+1}-Q_i),$$
where  we set $Q_{m+1}:=Q_1.$
\end{remark}

In the next example we consider the most important special case for $E,$ which is the only case that one has to consider for the proof of the theorem. 

\begin{example}\label{exampleSimplification}
Let us assume $E$ is the image of the diagonal embedding of $L_f$ in $M_m(D),$ i.e.
$$E=\{(x,\ldots,x)|\ x\in L_f\}.$$
Then $B$ and $j_E$ simplify, i.e.
\be
\item $B=\MopEnd_{\Delta}W$ with $\Delta:=\Centr{D}{L_f}$ and $W:=\bigoplus_iv_i\Delta$
\item In terms of lattice functions $j_E$ has the form
$$ j_E([\Lambda])=[\Lambda\cap W]$$
where $\Lambda\cap W$ denotes the lattice function 
$$ x\mapsto \Lambda(x)\cap W.$$
\item The image of $j_E|_{|\Sigma|}$ is the g.r. of the apartment $\Sigma_E$ which belongs to the frame $\{v_i\Delta|\ 1\leq i\leq m\}$ and in affine coordinates the map has the form
$$x\in\mathbb{R}^{m-1}\mapsto \frac{1}{f}x\in\mathbb{R}^{m-1}.$$
\item The vertices of $\Sigma_E$ are the points of $|\Sigma_E|$ with affine coordinate vectors in $\mathbb{Z}^{m-1}.$ Specifically the points $P_i:=j_E(Q_i)$ are vertices of a chamber of $\Sigma_E.$
\item  The edge from $P_i$ to $P_{i+1}$ is oriented to $P_{i+1}.$ 
\ee
\end{example}

\begin{proof}[example]
To prove the statements of the example it is enough to calculate $j_E$ in terms of lattice functions, i.e to show 2. The statements then follow by similar and standard calculations.\\
For 2: For an $o_D$-lattice function $\Lambda,$ whose class is $E^\times$-invariant there exists an $o_\Delta$-lattice
function $\Gamma,$ such that 
$$ j_E([\Lambda])=[\Gamma].$$
Using the decomposition 
$$ V=\tens{W}{\Delta}{D}=W\oplus W\pi_D\oplus W\pi_D^2\oplus\ldots\oplus W\pi_D^{f-1}, $$
the function 
$$\tilde{\Lambda}(x):=\bigoplus_{i=0}^{f-1}\Gamma(x-\frac{i}{f})\pi_D^{i}$$
is an $o_D$-lattice function of $V,$ such that
$$\textfrak{a}_{\Lambda}(x)\cap B=\textfrak{a}_{\tilde{\Lambda}}(x)\cap B.$$
Thus by the injectivity of $j_E$ the lattice functions $\Lambda$ and $\tilde{\Lambda}$ are equivalent and therefore 
$$ \Gamma=\tilde{\Lambda}\cap W\sim\Lambda\cap W .$$  
The appearence of $j_E$ in terms of coordinates follows now from
$$\textfrak{p}_D^{[x]+}\cap \Delta = \textfrak{p}_\Delta^{[\frac{[x]+}{f}]+}=\textfrak{p}_\Delta^{[\frac{x}{f}]+}.$$  
\end{proof}

\begin{proof}[theorem]
By lemma \ref{lemmaCommutativeDiagram} and by theorem \ref{thmPearlEmb}
we can assume that we are in the situation of the example \ref{exampleSimplification} above and that there is a diagonal matrix $h$ consisting of powers of $\pi_D$  with exponents in $\mathbb{N}_{f-1}\cup\{0\}$ such that 
$$(hEh^{-1},h\textfrak{a}h^{-1})$$
is the pearl embedding $(\lambda).$
We consider two cases for the proof.\\
{\bf Case 1:} $\textfrak{a}$ has period 1, i.e. 
$$h\textfrak{a}h^{-1}=M_m(o_D)=Q_1$$
and $\lambda$ is only one column.
We get $\textfrak{a}$ from $Q_1$ by applying $m_{h^{-1}}$ which is a composition of 
maps $m_g$ where $g$ differs from the identity matrix by only one diagonal entry $\pi_D^k.$ Now remark \ref{remarkDiagonalEmbedding} gives
$$\textfrak{a}=Q_1-\sum_{j=1}^{m}a_j(Q_{j+1}-Q_j)$$
where $a_j:=\frac{k-1}{f}$ if 
$$\sum_{i=1}^{k-1}\lambda_{i}<j\leq\sum_{i=1}^{k}\lambda_{i}.$$

Thus in barycentric coordinates $j_E(M_\textfrak{a})$ has the form
$$\frac{f-a_m+a_1}{f}P_1+\frac{a_2-a_1}{f}P_2+\ldots +\frac{a_{m}-a_{m-1}}{f}P_m.$$
and therefore the vector
$$\mu:=(\frac{f-a_m+a_1}{f},\frac{a_2-a_1}{f},\ldots,\frac{a_{m}-a_{m-1}}{f})$$
fullfils part one of the theorem.
If $(\lambda_{i_l})_{1\leq l\leq s}$ is the subsequence of non-zero entries we define the indexes 
$$j_l:=\lambda_1+\ldots+\lambda_{i_{l-1}}+1$$
and $j_1:=1.$
This are the indexes where the $\mu_j$ are non-zero, more precisely from 
$$j_l=\sum_{i=1}^{i_l-1}\lambda_i+1\leq \sum_{i=1}^{i_l}\lambda_i$$
we obtain for $a_j$ the following values:
$$a_j=a_{j_l}=i_l-1,\ j_l\leq j<j_{l+1}$$
and 
$$a_j=a_{j_s}=i_s-1,\ j_s\leq j\leq m,$$
and thus the subsequence of non-zero entries of $f\mu$ is 
$$(f\mu_{j_l})=(f-i_{s}+i_1,i_2-i_1,i_3-i_2,\ldots,i_{s}-i_{s-1}).$$
Therefore we get for $\pairs(\langle f\mu\rangle )$ the expression
$$\langle (1-i_{s}+i_1,\lambda_{i_1}),(i_2-i_1,\lambda_{i_2}),(i_3-i_2,\lambda_{i_3}),\ldots,(i_{s}-i_{s-1},\lambda_{i_s})\rangle $$
and this is precisely $\langle \row(\lambda)\rangle ^c.$\\
{\bf Case 2:} Assume the period $r$ of $\textfrak{a}$ is not 1. Here we want to use rank reduction. We fix an unramified field extension $L'|F$ of degree $r*d$ in an algebraic closure of $F.$
Denote by $D'$ a skewfield which is a central cyclic algebra over $F$ with maximal field $L'$ and an $L'$-normalising prime element $\pi_{D'},$ i.e.
$$D'=\bigoplus_{i=0}^{d*r-1}L'\pi_{D'}^i,$$ $$\pi_{D'}L'\pi_{D'}^{-1}=L',\ \text{and } \pi_{D'}^{d+r}=\pi_F.$$
The images of $L'_r,$ $L'_{r*f}$ under the diagonal embedding of $L'$ in $\Matr{m}{o_{D'}}$ are denoted by $F',$ $E'$ respectively and the apartment of the euclidean building ${\cal I}'$ of $GL_m(D')$ corresponding to the standart basis is denoted by $\Sigma',$ i.e.
we have a field tower
$$ E'\supseteq F'\supseteq F$$
and apartments $\Sigma',\ \Sigma'_{E'},\ \Sigma'_{F'}$
in the buildings ${\cal I}',\ {\cal I}'_{E'},\ {\cal I}'_{F'}$ respectively.
We then obtain a commutative diagram of bijections, where the lines are induced by isomorphisms of chamber complexes which preserve the orientation. 
\begin{displaymath}
\xymatrix{
|\Sigma| \ar[r]^{{\equiv}_F} \ar[d]^{j_E} & |\Sigma'_{F'}| \ar[d]^{j_{E'}} \\ |\Sigma_E|\ar[r]^{{\equiv}_E} & |\Sigma'_{E'}|
}
\end{displaymath}
The map ${\equiv}_F$ is given by
$$[x\mapsto \bigoplus_{i=0}^{m-1}v_i\textfrak{p}_D^{[d(x+\alpha_i)]+}]\mapsto [x\mapsto \bigoplus_{i=0}^{m-1}e_i\textfrak{p}_{\Centr{D'_r}{L'}}^{[d(x+\alpha_i)]+}]$$
and ${\equiv}_E$ analogously. Because of lemma \ref{lemmaChangingSkewfield} the map ${\equiv}_F$ preserves the embedding type and thus 
we can finish the proof by applying lemma \ref{LemRankReduction} on
$$\Sigma'\ra\Sigma'_{F'}\ra\Sigma'_{E'}.$$
More precisely, let $S_r$ be a face of rank $r$ in $\Sigma'_{F'}.$ Its barycenter has affine coordinates in $\frac{1}{r}\mathbb{Z}^{m-1}$ and therefore the preimage of it under $j_{F'}$ is a point $S_1$ with integer affine coeffitients, i.e. it corresponds to a vertex of ${\cal I}'.$ Because of 
$$j_{E'}(M_{S_r})=j_{E'}(j_{F'}(S_1))=j_{E'}(S_1)$$
the theorem follows now from the rank reduction lemma and case 1.
\end{proof}
\bibliographystyle{alpha}
\renewcommand\refname{Literature}
\bibliography{Bibliography}
\end{document}